\documentclass[reqno, 12pt,a4letter]{amsart}
\usepackage{amsmath,amsxtra,amssymb,latexsym, amscd,amsthm}
\usepackage[utf8]{inputenc}
\usepackage[mathscr]{euscript}
\usepackage{mathrsfs}
\usepackage[english]{babel}
\usepackage{color}

\newtheorem {theorem}{Theorem}[section]
\newtheorem {corollary}{Corollary}[section]
\newtheorem {lemma}{Lemma}[section]

\newtheorem {example}{Example}[section]
\newtheorem {definition}{Definition}[section]
\newtheorem {remark}{Remark}[section]
\newtheorem {proposition}{Proposition}[section]

\def\ar{a\kern-.370em\raise.16ex\hbox{\char95\kern-0.53ex\char'47}\kern.05em}
\def\ees{{\accent"5E e}\kern-.385em\raise.2ex\hbox{\char'23}\kern-.08em}
\def\eex{{\accent"5E e}\kern-.470em\raise.3ex\hbox{\char'176}}
\def\AR{A\kern-.46em\raise.80ex\hbox{\char95\kern-0.53ex\char'47}\kern.13em}
\def\EES{{\accent"5E E}\kern-.5em\raise.8ex\hbox{\char'23 }}
\def\EEX{{\accent"5E E}\kern-.60em\raise.9ex\hbox{\char'176}\kern.1em}
\def\ow{o\kern-.42em\raise.82ex\hbox{
  \vrule width .12em height .0ex depth .075ex \kern-0.16em \char'56}\kern-.07em}
\def\OW{O\kern-.460em\raise1.36ex\hbox{
\vrule width .13em height .0ex depth .075ex \kern-0.16em \char'56}\kern-.07em}
\def\UW{U\kern-.42em\raise1.36ex\hbox{
\vrule width .13em height .0ex depth .075ex \kern-0.16em \char'56}\kern-.07em}
\def\DD{D\kern-.7em\raise0.4ex\hbox{\char '55}\kern.33em}

\title[]{Complementary problems with polynomial data}

\author{TI\EES N-S\OW n PH\d{A}M}
\address{Department of Mathematics, University of Dalat, 1 Phu Dong Thien Vuong, Dalat, Vietnam}
\email{sonpt@dlu.edu.vn}
              
\author{C\AR NH H\`UNG NGUY\EEX N}
\address{Department of Mathematics, Nha Trang University, 02 Nguyen Dinh Chieu, Nha Trang, Vietnam}
\email{hungnc@ntu.edu.vn}

\dedicatory{Dedicated to Professor Boris Mordukhovich on the occasion of his 70th birthday}

\thanks{The first author is partially supported by Vietnam National Foundation for Science and Technology Development (NAFOSTED)} 



\date{ \today}

\begin{document}
\maketitle

\begin{abstract}
Given polynomial maps $f, g \colon \mathbb{R}^n \to \mathbb{R}^n,$ we consider the {\em polynomial complementary problem} of finding a vector $x \in \mathbb{R}^n$ such that
\begin{equation*}
f(x) \ \ge \ 0, \quad g(x) \ \ge \  0, \quad \textrm{ and } \quad \langle f(x), g(x) \rangle \ = \ 0.
\end{equation*}
In this paper, we present various properties on the solution set of the problem, including genericity, nonemptiness, compactness, uniqueness as well as error bounds with exponents explicitly determined. These strengthen and generalize some previously known results, and hence broaden the boundary knowledge of nonlinear complementarity problems as well.
\end{abstract}

\section{Introduction}

We consider the {\em polynomial complementary problem} (PCP) of finding a vector $x \in \mathbb{R}^n$ such that
\begin{equation}\label{PCP}
f(x) \ \ge \ 0, \quad g(x) \ \ge \  0, \quad \textrm{ and } \quad \langle f(x), g(x) \rangle \ = \ 0,
\end{equation}
where $f, g \colon \mathbb{R}^n \to \mathbb{R}^n$ are polynomial maps. We denote this problem by $\mathrm{PCP}(f, g)$ for short.
This is just the nonlinear complementarity problem and reduces to the classical {\em linear complementarity problem} (LCP) when $f$ is the identity map and $g$ is an affine map. With an extensive theory, algorithms, and applications, the linear complementarity problem has been well studied in the optimization literature; for more details, we refer the reader to the comprehensive monographs \cite{Cottle2009, Facchinei2003} with the references therein. Note, too, that PCPs contain {\em tensor complementarity problems} which has received considerable attention in recent years, see e.g., \cite{Bai2016, Fan2018, Li2015, Qi2018, Song2016}.

Let $\mathrm{SOL}(f, g)$ denote the solution set of the $\mathrm{PCP}(f, g)$ and define the {\em natural} map $\frak{m} \colon \mathbb{R}^n \to \mathbb{R}^n$ by
$$\frak{m}(x) := \min\{f(x), g(x)\},$$
where the $\min$ operator denotes the componentwise minimum of two vectors. Then it is clear that $\mathrm{SOL}(f, g)$ is precisely the zero set of the natural map $\frak{m}.$ In other words, we have
$$\mathrm{SOL}(f, g) = \{x \in \mathbb{R}^n \ : \ \frak{m}(x) = 0\}.$$

We first assume that $f$ is the {\em identity} map ${Id}.$ The literature on PCPs (in particular, LCPs) in this case is vast and we confine ourselves to quoting a few that are relevant to our study. Some generic properties of complementarity problems are known; for example, 
Saigal and Simon \cite{Saigal1973} have shown that for almost all maps $g,$ the corresponding complementarity problem has a discrete solution set.

The existence, boundedness, and uniqueness of solutions of LCPs are well-studied topics; see \cite{Cottle2009, Facchinei2003}. 
Very recently, Gowda \cite{Gowda2017} (see also \cite{Gowda2018, Hu2018-1}) establishes results connecting the polynomial complementarity problem $\mathrm{PCP}(Id, g)$ and the tensor complementarity problem $\mathrm{PCP}(Id, g^\infty),$ where $g^\infty$ is the homogeneous part of highest degree of $g.$ In particular, he shows some properties on the solution set of PCPs, including nonemptiness, boundedness and uniqueness. 

An important topic in the study of complementarity problems concerns {\em error bounds} for estimating the distance from an arbitrary point $x \in \mathbb{R}^n$ to the solution set $\mathrm{SOL}(Id, g)$ in terms of the natural map. When $g$ is an {\em affine} map, it is well-known that a local Lipschitzian error bound holds due to Robinson \cite{Robinson1981} (see also Theorem~\ref{Theorem12} in Section~\ref{Section5} for a different proof) and, under some assumptions, a global Lipschitzian error bound holds; for more details, we refer the reader to Chapter~6 in the monograph \cite{Facchinei2003} by Facchinei and Pang with the references therein. Very recently, assume that $g$ is a {\em quadratic} map satisfying a certain additional condition, Hu, Wang and Huang \cite{Hu2018-2} derive some local H\"olderian error bound results with explicit exponents.

However, there has, to the best of our knowledge, been no attempt to extend the results mentioned above to the case where $f$ is {\em not} the identity map. In this paper, we undertake this study for the $\mathrm{PCP}(f, g)$ with $f$ and $g$ being {\em arbitrary} polynomial maps. Our main contributions are as follows: 
\begin{enumerate}
\item[(i)] We show that the solution set of PCPs is finite generically. 
Here and in the following, we say that a given property holds generically, if it holds in an open and dense (semialgebraic) set of the entire space of input data.

\item[(ii)] We provide a necessary and sufficient condition for the compactness of the solution set of the $\mathrm{PCP}(f, g)$ in terms of the natural map $\frak{m}.$

\item[(iii)] We show that under appropriate conditions, the $\mathrm{PCP}(f + p, g + q)$ has a nomempty compact solution set for all polynomial maps $p$ and $q$ of degrees less than those of $f$ and $g,$ respectively.

\item[(iv)] We establish some (local and global) H\"olderian error bound results for the solution set of the $\mathrm{PCP}(f, g)$ in terms of the natural map $\frak{m}$ with exponents explicitly determined by the dimension $n$ of the underlying space $\mathbb{R}^n$ and the degree of the involved polynomial maps $f$ and $g.$
Furthermore, it is shown that, generically, PCPs have a global Lipschitzian error bound.

\end{enumerate}

Consequently, the results presented in this paper strengthen and generalize some previously known results, and hence broaden the boundary knowledge of nonlinear complementarity problems as well.

The rest of the paper is organized as follows. Section~\ref{Section2} covers some preliminary materials. In Section~\ref{Section3}, genericity properties for PCPs are addressed. In Section~\ref{Section4}, various properties on solution sets for PCPs, including the nonemptiness, compactness and uniqueness are presented. Finally, in Section~\ref{Section5}, error bound results via the natural map with explicit exponents is established.

\section{Preliminaries} \label{Section2}

\subsection{Notations}
We shall use the following notations throughout the paper. Fix a number  $n \in {\Bbb N}$, $n \geqq 1$, and abbreviate $(x_1, \ldots, x_n)$ by $x.$  The space $\mathbb{R}^n$ is equipped with the usual scalar product $\langle \cdot, \cdot \rangle$ and the corresponding Euclidean norm $\| \cdot\|.$ As usual, $\mathrm{dist}(x, S)$ denotes the Euclidean distance from $x \in \mathbb{R}^n$ to $S \subset \mathbb{R}^n,$ i.e.,
\begin{eqnarray*}
\mathrm{dist}(x, S) &:=& \inf\{\|x - y \| \ : \ y \in S \},
\end{eqnarray*}
where, by convention, the infimum is $1$ if $S$ is empty.

For a vector $x := (x_1, \ldots, x_n) \in \mathbb{R}^n$ we write $x \ge 0$ when $x_i \ge 0$ for all $i = 1, \ldots, n.$

For two vectors $x$ and $y$ in $\mathbb{R}^n,$ we write $\min\{x, y\}$ for the vector whose $i$th component is $\min\{x_i, y_i\}.$ Observe that
$$\min\{x, y\} = 0 \quad \Leftrightarrow \quad \quad x \ge 0, y \ge 0, \ \textrm{ and } \ \langle x, y \rangle = 0.$$

A polynomial map $f \colon \mathbb{R}^n \to \mathbb{R}^n$ is {\em homogeneous} of degree $d$ (which is a natural number) if $f(tx) = t^d f(x)$ for all $x \in \mathbb{R}^n$ and all $t \in \mathbb{R}.$

Consider a polynomial map $f \colon \mathbb{R}^n \to \mathbb{R}^n,$ which is expressed, after regrouping terms, in the following
form:
\begin{eqnarray*}
f(x) &=& \mathcal{A}_d(x) + \mathcal{A}_{d - 1}(x) + \cdots + \mathcal{A}_0(x),
\end{eqnarray*}
where each term $\mathcal{A}_k$ is a polynomial map, homogeneous of degree $k.$ We assume that $\mathcal{A}_d$ is nonzero and say that $f$ is a polynomial map of degree $d.$ Let $f^\infty := \mathcal{A}_d$ denote the ``leading term'' of $f.$

\subsection{The notion of degree}


In this paper we use systematically the topological degree as it is presented in \cite[Chapter~6]{Cottle2009}, \cite[Chapter~2]{Facchinei2003} and \cite{Lloyd1978}. Here is a short review of what we need in our development. 

Suppose $\Omega$ is a bounded open set in $\mathbb{R}^n,$ $f \colon \overline{\Omega} \rightarrow \mathbb{R}^n$ is a continuous map, and $y \not \in f(\partial \Omega),$ where $\overline{\Omega}$ and $\partial \Omega$ denote, respectively, the closure and boundary of $\Omega.$ Then an integer called {\em the degree of $f$ at $y$ relative to} $\Omega$ is defined. This number, denoted by  $\deg(f, \Omega, y),$ gives an estimation and the nature of the solution(s) of the equation $f(x) = y$ in $\Omega.$ When this degree is nonzero, the equation $f(x) = y$ has a solution in $\Omega.$ Suppose $f(x) = y$ has a unique solution, say, $x^*$ in $\Omega.$ Then, $\deg (f, \Omega', y)$ is constant over all bounded open sets $\Omega'$ containing $x^*$ and contained in $\Omega.$
This common degree is called the {\em local (topological) degree} of $f$ at $x^*$ (also called the {\em index} of $f$ at $x^*$ in some literature); 
it will be denoted by $\deg (f, x^*).$ In particular, if $g \colon \mathbb{R}^n \rightarrow \mathbb{R}^n$ is a continuous map such that $g(x) = y \Rightarrow x = x^*,$ then, for any bounded open set $\Omega$ in $\mathbb{R}^n$ containing $x^*,$ we have
\begin{eqnarray*}
\deg (g, x^*) & = & \deg (g, \Omega, y);
\end{eqnarray*}
moreover, 
$\deg (g, x^*) = 1$ (resp.,  $\deg (g, x^*) = \pm 1$) when $g$ is the identity map (resp., a homeomorphism map).

Let $H \colon \mathbb{R}^n \times [0, 1] \rightarrow \mathbb{R}^n, (x, t) \mapsto H(x, t),$ be a continuous map (in which case, we say that $H$ is a {\em homotopy}) and assume that the zero set 
\begin{eqnarray*}
\{x \in \mathbb{R}^n \ : \ H(x, t) = 0 \quad \textrm{ for some } \quad t \in [0, 1]\}
\end{eqnarray*}
is bounded. Then, for any bounded open set $\Omega$ in $\mathbb{R}^n$ that contains this zero set, we have the homotopy invariance property of degree:
\begin{eqnarray*}
\deg (H(\cdot, 1), \Omega, 0) \ = \ \deg (H(\cdot, 0), \Omega, 0).
\end{eqnarray*}

\subsection{Error bounds for polynomial systems}

In this subsection we recall an error bound result with explicit exponents for polynomial systems over compact sets. 

To state the result, let us begin with some notation. Given a real number $a,$ we define $[-a]_+ := \max\{-a, 0\}$ as usual, so $a \ge 0$ if, and only if, $[-a]_+ = 0.$ 
Following D'Acunto and Kurdyka \cite{Acunto2005}, for two positive integer numbers $n$ and $d,$ we let
\begin{eqnarray*}
\mathscr{R}(n, d) :=
\begin{cases}
d(3d - 3)^{n - 1} &\ {\rm if} \ d \geq 2, \\
1 &\  {\rm if} \ d = 1.
\end{cases}
\end{eqnarray*}
The following result will be used in Section~\ref{Section5}.

\begin{lemma}\label{Lemma21}
Let $g_i$ as $i=1,\ldots, l$ and $h_j$ as $j=1,\ldots, m,$ be real polynomials on $\mathbb{R}^n$ of degrees at most $d$, and let 
\begin{equation*}
S := \{ x \in \mathbb{R}^n \ | \ g_1(x) = 0, \ldots, g_l(x) = 0, \ h_1(x) \le 0, \ldots, h_m(x) \le 0\} \ne \emptyset.
\end{equation*}
Then for any compact set $K \subset \mathbb{R}^n,$ there is a constant $c > 0$ such that
\begin{eqnarray*}
c\, \mathrm{dist}(x, S)^{\mathscr{R}(n + l + m - 1, d + 1)} &\le& \sum_{i=1}^l|g_i(x)| + \sum_{j=1}^m [h_j(x)]_+  \quad \textrm{ for all } \quad x \in K.
\end{eqnarray*}
\end{lemma}

\begin{proof}
See \cite[Theorem~3.3]{HaHV2017} or \cite[Theorem~3.5]{Li2015}.
\end{proof}

\section{Generic properties}\label{Section3}

In this section we show that for a generic set of polynomial maps $(f, g),$ the corresponding complementarity problem $\mathrm{PCP}(f, g)$ has a finite solution set. 
To this end,  we fix some notation. Given positive integers $d_1, \ldots, d_n,$ let $\mathbf{P}_{(d_1, \ldots, d_n)}$ denote the set of polynomial maps $f := (f_1, \ldots, f_n)$ from $\mathbb{R}^n$ to itself with $\deg f_i \le d_i$ for $i = 1, \ldots, n.$  If $\kappa = (\kappa_1, \ldots, \kappa_n) \in \mathbb{N},$ we denote by $x^\kappa$ the monomial  $x_1^{\kappa_1} \cdots x_n^{\kappa_n}$ and by $| \kappa|$ the sum $\kappa_1 + \cdots + \kappa_n.$  For each $i \in \{ 1, \ldots, n\},$ by using the lexicographic ordering on the set of monomials $x^\kappa, |\kappa| \le d_i,$ we  may identify each polynomial function $f_i(x) := \sum_{|\kappa| \le d_i} u_{i, \kappa} x^\kappa,$ with its vector of coefficients, i.e.,  
$f_i  \equiv (u_{i, \kappa})_{|\kappa| \le d_i}  \in\mathbb{R}^{s_i},$ where $s_i := \# \{ \kappa \in \mathbb{N}^n \ : \ |\kappa| \le d_i\}.$ Then $\mathbf{P}_{(d_1, \ldots, d_n)}$ is identified with the Euclidean space $\mathbb{R}^{s_1} \times \cdots \times \mathbb{R}^{s_n}.$

The following result is inspired by the work of Saigal and Simon \cite{Saigal1973}.

\begin{proposition} \label{Proposition31}
Let $d_i, d_i', i = 1, \ldots, n,$ be positive integer numbers. There exists an open dense semi-algebraic set $\mathscr{U}$ in $\mathbf{P}_{(d_1, \ldots, d_n)} \times \mathbf{P}_{(d_1', \ldots, d_n')}$ such that for all $(f, g) \in \mathscr{U},$ the following statements hold:
\begin{enumerate}
\item[{\rm (i)}] The solution set for the corresponding complementarity problem $\mathrm{PCP}(f, g)$ is finite and has at most $(2d)^n$ elements, where $d := \max_{i = 1, \ldots, n} \{d_i, d_i'\}.$ 

\item[{\rm (ii)}] If $x$ is a solution of $\mathrm{PCP}(f, g)$ then $f_i(x) + g_i(x) > 0$ for all $i = 1, \ldots, n.$

\item[{\rm (iii)}] $\mathrm{SOL}(f^\infty, g^\infty)  = \{0\}.$
\end{enumerate}
\end{proposition}

\begin{proof}
Indeed, by definition, we have
\begin{eqnarray*}
\mathrm{SOL}(f, g) &\subset& \bigcup_{I} \{x \in \mathbb{R}^n \ : \ f_i(x) = 0 \textrm{ for } i \in I \textrm{ and }
g_i(x) = 0 \textrm{ for } i \not \in I\},
\end{eqnarray*}
where the union is taken over all subsets $I$ of the set $\{1, \ldots, n\}.$ Clearly, the desired conclusion follows immediately from the next lemma.
\end{proof}

\begin{lemma}
For each index set $I \subset \{1, \ldots, n\},$ there exists an open dense semi-algebraic set $\mathscr{U}_I \subset \mathbf{P}_{(d_1, \ldots, d_n)} \times \mathbf{P}_{(d_1', \ldots, d_n')}$ such that for all $(f, g) \in \mathscr{U}_I,$ the following two assertions hold:
\begin{enumerate}
\item[{\rm (i)}] The set
$$\{x \in \mathbb{R}^n \ : \ f_i(x) = 0 \textrm{ for } i \in I \textrm{ and }
g_i(x) = 0 \textrm{ for } i \not \in I\}$$
is finite and has at most $\prod_{i \in I} d_i \times \prod_{i \not \in I'} d_i'$ elements.

\item[{\rm (ii)}] For all $h \in \{g_i \textrm{ for } i \in I \textrm{ and } f_i \textrm{ for } i \not \in I\},$ the set
$$\{x \in \mathbb{R}^n \ : \ f_i(x) = 0 \textrm{ for } i \in I, \
g_i(x) = 0 \textrm{ for } i \not \in I, \textrm{ and } h(x) = 0\}$$
is empty.

\item[{\rm (iii)}] The system of homogeneous equations
$$f_i^\infty(x) = 0 \textrm{ for } i \in I \textrm{ and } g_i^\infty(x) = 0 \textrm{ for } i \not \in I$$
has a unique solution $x = 0.$
\end{enumerate}

\end{lemma}

\begin{proof}
Without loss of generality we may assume that $I = \{1, \ldots, n\}.$

(i) Consider the polynomial map
$$ \Phi \colon  \mathbf{P}_{(d_1, \ldots, d_n)} \times \mathbb{R}^n \rightarrow  \mathbb{R}^n, \quad  (f, x) \mapsto  f(x).$$
For each polynomial map $f := (f_1, \ldots, f_n) \in \mathbf{P}_{(d_1, \ldots, d_n)},$ let us write $f_i(x) = \sum_{|\kappa| \le d_i} u_{i, \kappa} x^\kappa$ for $i = 1, \ldots, n.$
Since each polynomial $f_i$ is identified with its vector of coefficients $(u_{i, \kappa})_{|\kappa| \le d_i},$ a simple calculation shows that
$$\left( \frac{\partial\Phi_{j}}{\partial u_{i, \kappa}}\right) _{|\kappa| = 0, \, i, j = 1, 2, \ldots, n}$$
is the unit matrix of order $n,$ and so the Jacobian of $\Phi$ has rank $n$ at every point in $\mathbf{P}_{(d_1, \ldots, d_n)} \times \mathbb{R}^n.$  In particular, $0 \in \mathbb{R}^n$ is a regular value of $\Phi.$ By the Sard theorem with parameter (see, for example, \cite[Theorem~1.10]{HaHV2017}), there exists an open dense semi-algebraic set $\mathscr{V} \subset \mathbf{P}_{(d_1, \ldots, d_n)}$ such that for all $f \in \mathscr{V},$ either the set $\Phi_f^{-1}(0)$ is empty or for any $x \in \Phi_f^{-1}(0),$ the derivative map
$$D_x \Phi_f \colon  \mathbb{R}^n \rightarrow  \mathbb{R}^n$$
is surjective, where $\Phi_f$ denotes the map
$$\Phi_f \colon  \mathbb{R}^n \rightarrow  \mathbb{R}^n, \quad x \mapsto f(x).$$
By the inverse function theorem, then the zero set $\Phi_f^{-1}(0),$ which is the set $f^{-1}(0) ,$ is discrete, and so is finite (possibly empty). Furthermore, since the Jacobian $\left(\frac{\partial f_i}{\partial x_j}(x) \right) _{i, j = 1, 2, \ldots, n}$ has rank $n$ at every point $x$ in $f^{-1}(0),$ it follows from Bezout's theorem (cf., for instance, \cite[Appendix~B]{Benedetti1990} or \cite[Chapter~9]{Bochnak1998}) that the set $f^{-1}(0)$ has at most $\prod_{i = 1}^n d_i$ elements. 

(ii) Fix $i \in I = \{1, \ldots, n\}.$ We will show that there exists an open dense semi-algebraic set $\mathscr{U}^{(i)} \subset \mathbf{P}_{(d_1, \ldots, d_n)} \times \mathbf{P}_{(d_1', \ldots, d_n')}$ such that for each $(f, g) \in \mathscr{U}^{(i)},$ the polynomial equations
$$f_1(x) = 0, \ldots, f_n(x) = 0, g_i(x) = 0$$
have no common solutions. To this end, consider the polynomial map
$$ \Psi \colon  \mathbf{P}_{(d_1, \ldots, d_n)} \times \mathbf{P}_{(d_1', \ldots, d_n')}
\times \mathbb{R}^n \rightarrow  \mathbb{R}^n \times \mathbb{R}, \quad  (f, g, x) \mapsto  (f(x), g_i(x)).$$
Following the same procedure as in (i), we obtain an open dense semi-algebraic subset $\mathscr{U}^{(i)}$ of $\mathbf{P}_{(d_1, \ldots, d_n)} \times \mathbf{P}_{(d_1', \ldots, d_n')}$ such that for all $(f, g) \in \mathscr{U}^{(i)},$ either
$\Psi_{(f, g)}^{-1}(0) = \emptyset$ or for any $x \in \Psi_{(f, g)}^{-1}(0),$ the derivative map
$$D_x \Psi_{(f, g)} \colon \mathbb{R}^n \rightarrow  \mathbb{R}^n \times \mathbb{R}$$
is surjective, where $\Psi_{(f, g)} $ stands for the map
$$\mathbb{R}^n \rightarrow  \mathbb{R}^n \times \mathbb{R}, \quad x \mapsto \Psi(f, g, x).$$
But the latter case is impossible since $\dim \mathbb{R}^{n} = n < n + 1 = \dim (\mathbb{R}^{n} \times \mathbb{R}).$ So, $\Psi_{(f, g)}^{-1}(0) = \emptyset;$ i.e., the solution set of
$$f_1(x) = 0, \ldots, f_n(x) = 0, g_i(x) = 0$$
is empty.

(iii) Let $\mathbf{H}_{(d_1, \ldots, d_n)}$ denote the set of polynomial maps $f := (f_1, \ldots, f_n) \in \mathbf{P}_{(d_1, \ldots, d_n)}$ such that each component $f_i$ is homogeneous of degree $d_i.$ For each $i \in \{ 1, \ldots, n\},$ by using the lexicographic ordering on the set of monomials $x^\kappa, |\kappa|  = d_i,$ we  may identify each homogeneous polynomial $f_i(x) := \sum_{|\kappa|  = d_i} u_{i, \kappa} x^\kappa$ with its vector of coefficients, i.e.,  
$f_i  \equiv (u_{i, \kappa})_{|\kappa|  = d_i}  \in\mathbb{R}^{t_i},$ where $t_i := \# \{ \kappa \in \mathbb{N}^n \ : \ |\kappa|  = d_i\}.$ Then $\mathbf{H}_{(d_1, \ldots, d_n)}$ is identified with the Euclidean space $\mathbb{R}^{t_1} \times \cdots \times \mathbb{R}^{t_n}.$

Let $\mathbb{S}^{n - 1}$ denote the unit sphere in $\mathbb{R}^{n}$ and consider the polynomial map
$$\Gamma \colon  \mathbf{H}_{(d_1, \ldots, d_n)} \times \mathbb{S}^{n - 1} \rightarrow  \mathbb{R}^n, \quad  (f, x) \mapsto  f(x).$$
Take any $f := (f_1, \ldots, f_n) \in \mathbf{H}_{(d_1, \ldots, d_n)}$ and assume $f_i \equiv (u_{i, \kappa})_{|\kappa| = d_i}$ for $i = 1, \ldots, n.$ Take any $x \in \mathbb{S}^{n - 1}.$
Without loss of generality, we may assume that $x_1 \ne 0.$ A simple calculation shows that
\begin{eqnarray*}
\left( \frac{\partial\Gamma_{j}}{\partial u_{i, \kappa^i}}\right) _{\kappa^i = (d_i, 0, \ldots, 0), \, i, j = 1, 2, \ldots, n} &=&
\begin{pmatrix}
x_1^{d_1} & 0 & \cdots & 0 \\
0 & x_1^{d_2} & \cdots & 0 \\
\vdots & \vdots & \ddots & \vdots \\
0 & 0 & \cdots & x_1^{d_n}
\end{pmatrix}.
\end{eqnarray*}
Consequently, the Jacobian of $\Gamma$ has rank $n$ at $(f, x) \in \mathbf{H}_{(d_1, \ldots, d_n)} \times \mathbb{S}^{n - 1}.$
In particular, $0 \in \mathbb{R}^n$ is a regular value of $\Gamma.$ By the Sard theorem with parameter (see, for example, \cite[Theorem~1.10]{HaHV2017}), there exists an open dense semi-algebraic set $\mathscr{W} \subset \mathbf{H}_{(d_1, \ldots, d_n)}$ such that for all $f \in \mathscr{W},$ either the set $\Gamma_f^{-1}(0)$ is empty or for any $x \in \Gamma_f^{-1}(0),$ the derivative map
$$D_x \Gamma_f \colon T_x \mathbb{S}^{n - 1} \rightarrow  \mathbb{R}^n$$
is surjective, where $T_x \mathbb{S}^{n - 1}$ is the tangent space of the sphere $\mathbb{S}^{n - 1}$ at $x$ and $\Gamma_f$ denotes the map
$$\Gamma_f \colon \mathbb{S}^{n - 1} \rightarrow  \mathbb{R}^n, \quad x \mapsto f(x).$$
But the latter case is impossible since $\dim \mathbb{S}^{n - 1} = n - 1 < n = \dim \mathbb{R}^{n}.$ So, $\Gamma_{f}^{-1}(0) = \emptyset$ and hence (by homogeneity of the polynomial functions $f_i$)
\begin{eqnarray*}
\{x \in \mathbb{R}^n \ : \ f_1(x) = 0, \ldots, f_n(x) = 0\} &=& \{0\}.
\end{eqnarray*}

Now we can see that the semi-algebraic set
\begin{eqnarray*}
\mathscr{U}^{(n + 1)} &:=& \{(f, g) \in \mathbf{P}_{(d_1, \ldots, d_n)} \times \mathbf{P}_{(d_1', \ldots, d_n')} \ : \ f^\infty \in \mathscr{W}\}.
\end{eqnarray*}
is open dense in $\mathbf{P}_{(d_1, \ldots, d_n)} \times \mathbf{P}_{(d_1', \ldots, d_n')},$ and furthermore, for every $(f, g) \in \mathscr{U}^{(n + 1)}$ the system of homogeneous equations
$$f_1^\infty(x) = 0, \ldots, f_n^\infty(x) = 0,$$
has a unique solution $x = 0.$

\medskip
Finally, it is easy to check that the set 
$$\mathscr{U}_I := (\mathscr{V} \times \mathbf{P}_{(d_1', \ldots, d_n')}) \cap \mathscr{U}^{(1)} \cap \cdots \cap \mathscr{U}^{(n)} \cap \mathscr{U}^{(n + 1)}$$ 
has the desired properties.
\end{proof}

\section{Nonemptiness, compactness, and uniqueness} \label{Section4}

In this section, various properties on solution sets for polynomial complementarity problems, including nonemptiness, compactness, and uniqueness are established. Some results presented below extend those of Gowda~\cite{Gowda2017}
and Karamardian~\cite{Karamardian1972, Karamardian1976}.

Given two polynomial maps $f, g$ from $\mathbb{R}^n$ to itself, recall that the {\em solution set} of the $\mathrm{PCP}(f, g)$ is denoted by $\mathrm{SOL}(f, g),$ i.e., 
\begin{eqnarray*}
\mathrm{SOL}(f, g) &:=& \{x \in \mathbb{R}^n \ : \ f(x) \ge 0, \ g(x) \ge 0, \ \langle f(x), g(x) \rangle = 0\}.
\end{eqnarray*}
By definition, $\mathrm{SOL}(f, g)$ is a closed and semialgebraic set, and so it has finitely many connected components (see, for example, \cite{Bochnak1998}). Recall also that the {\em natural} map $\frak{m} \colon \mathbb{R}^n \to \mathbb{R}^n$ is given by
\begin{eqnarray*}
\frak{m}(x) &:=& \min\{f(x), g(x)\}.
\end{eqnarray*}
It is easy to see that the map $\mathfrak{m}$ is locally Lipschitz and semialgebraic, and satisfies the relation
\begin{eqnarray*}
\mathrm{SOL}(f, g) &=& \{x \in \mathbb{R}^n \ : \ \frak{m}(x) = 0\}.
\end{eqnarray*}

\begin{proposition}\label{MD1}
Let $f, g \colon \mathbb{R}^n \to \mathbb{R}^n$ be polynomial maps. The following conditions are equivalent:
\begin{itemize}
\item[(i)] The set $\mathrm{SOL}(f, g)$ is compact (possibly empty).

\item[(ii)] There exist constants $c > 0, R > 0$, and $\alpha \in \mathbb{Q}$ such that
\begin{eqnarray*}
\|\frak{m}(x) \| &\ge& c \|x\|^\alpha \quad \textrm{ for all } \quad \|x\| \ge R.
\end{eqnarray*}
\end{itemize}
\end{proposition}
\begin{proof}
Clearly, it suffices to show the implication (i) $\Rightarrow$ (ii). To this end, define the function $\varphi \colon [0, +\infty) \to \mathbb{R}$ by
\begin{eqnarray*}
\varphi(t) &:=& \min_{\|x\| = t} \|\frak{m}(x) \|.
\end{eqnarray*}
Then $\varphi$ is non-negative and $\varphi(t) > 0$ for all $t > 0$ sufficiently large. Furthermore, by the Tarski--Seidenberg theorem \cite[Theorem~1.5]{HaHV2017}, the function $\varphi$ is semi-algebraic. 
So, thanks to the monotonicity theorem \cite[Theorem~1.8]{HaHV2017}, we can find a constant $R > 0$ such that $\varphi$ is either constant or strictly increasing or strictly  decreasing on $[R, +\infty).$ 

If the function $\varphi$ is constant on $[R, +\infty),$ say $c,$ then we have for all $x \in \mathbb{R}^n$ with $\|x\| \ge R,$
\begin{eqnarray*}
\|\frak{m}(x) \| &\ge& \min_{\|u\| = \|x\|} \|\frak{m}(u) \|  \ = \ \varphi(\|x\|) \ = \ c,
\end{eqnarray*}
establishing (ii) with the exponent $\alpha = 0.$

Assume that the function $\varphi$ is not constant on $[R, +\infty).$ In view of the growth dichotomy lemma \cite[Lemma~1.7]{HaHV2017}, we can write 
\begin{eqnarray*}
\varphi(t) &=& a\, t^\alpha + \textrm{ terms of lower degree}
\end{eqnarray*}
for some $a > 0$ and $\alpha \in \mathbb{Q}.$ Consequently, there exists a constant $c > 0$ such that
\begin{eqnarray*}
\varphi(t) &\ge& c\, t^\alpha \quad \textrm{ for all } \quad t \ge R
\end{eqnarray*} 
(after perhaps increasing $R$). Now we have for all $x \in \mathbb{R}^n$ with $\|x\| \ge R,$
\begin{eqnarray*}
\|\frak{m}(x) \| &\ge& \min_{\|u\| = \|x\|} \|\frak{m}(u) \|  \ = \ \varphi(\|x\|) \ \ge \ c \|x\|^\alpha,
\end{eqnarray*}
which completes the proof of (i) $\Rightarrow$ (ii). 
\end{proof}

\begin{proposition}\label{MD3}
Let $f, g \colon \mathbb{R}^n \to \mathbb{R}^n$ be polynomial maps. The following conditions are equivalent:
\begin{itemize}
\item[(i)] The set $\mathrm{SOL}(f, g)$ is unbounded.

\item[(ii)] The set
\begin{eqnarray*}
\{x \in \mathbb{R}^n \ : \ f(x) \ge 0, \ g(x) \ge 0, \ \langle f(x), g(x) \rangle = 0, \ \min_{I} |\det \mathrm{Jac}_I(x)| = 0 \}
\end{eqnarray*}
is unbounded, where for each index set $I \subset \{1, \ldots, n\},$ $\mathrm{Jac}_I$ stands for the Jacobian of the map 
$$\mathbb{R}^n \to \mathbb{R}^n, \quad x \mapsto (f_i(x), g_j(x))_{i \in I, j \not \in I}.$$
\end{itemize}
\end{proposition}
\begin{proof}
(ii) $\Rightarrow$ (i): Obviously.

(i) $\Rightarrow$ (ii): Assume that $\mathrm{SOL}(f, g)$ is unbounded. By the curve selection lemma at infinity \cite[Theorem~1.12]{HaHV2017},  we can find an analytic curve $\phi \colon (0, \epsilon) \to \mathbb{R}^n$ such that $\|\phi(t)\| \to +\infty$ as $t \to 0^+$ and
$\phi(t) \in \mathrm{SOL}(f, g)$ for all $t \in (0, \epsilon) .$ Then for each $i \in \{1, \ldots, n\},$ we have $(f_i \circ \phi)(t) \cdot (g_i \circ \phi)(t)  = 0$ for all $t \in (0, \epsilon).$ By the monotonicity theorem \cite[Theorem~1.8]{HaHV2017}, we can assume that the functions $f_i \circ \phi$ and $g_i \circ \phi$ are either constant or strictly increasing or strictly  decreasing on $(0, \epsilon)$ (after perhaps shrinking $\epsilon$). Hence, either $f_i \circ \phi \equiv 0$ or $g_i \circ \phi \equiv 0.$ Therefore, there exists an index set $I \subset \{1, \ldots, n\}$ such that
\begin{eqnarray*}
f_i \circ \phi \equiv 0 \textrm{ for } i \in I \quad \textrm{ and } \quad 
g_j \circ \phi \equiv 0 \textrm{ for } j \not \in I.
\end{eqnarray*}
Consequently, we have for all $t \in (0, \epsilon),$
\begin{eqnarray*}
\langle \nabla f_i(\phi(t)), \frac{d \phi(t)}{dt}\rangle  = 0 \textrm{ for } i \in I \quad \textrm{ and } \quad 
\langle \nabla g_j(\phi(t)), \frac{d \phi(t)}{dt}\rangle   = 0 \textrm{ for } j \not \in I,
\end{eqnarray*}
or equivalently, 
\begin{eqnarray} \label{PT3}
\mathrm{Jac}_I (\phi(t)) \frac{d \phi(t)}{dt} &=& 0.
\end{eqnarray}
On the other hand, since $\|\phi(t)\| \to +\infty$ as $t \to 0^+,$ the monotonicity theorem \cite[Theorem~1.8]{HaHV2017} gives us that the function $(0, \epsilon) \to \mathbb{R}, t \mapsto \|\phi(t)\|^2,$ is strictly increasing (after perhaps shrinking $\epsilon$ again). 
We deduce that for all $t > 0$ sufficiently small, $\frac{d \|\phi(t)\|^2}{dt} > 0$ and hence that $\frac{d \phi(t)}{dt} \ne 0.$ 
It follows immediately from~\eqref{PT3} that $\det \mathrm{Jac}_I(\phi(t)) = 0$ for  $t > 0$ sufficiently small, which completes the proof.
\end{proof}

\begin{proposition}\label{MD2}
Let $f, g \colon \mathbb{R}^n \to \mathbb{R}^n$ be polynomial maps of positive degrees $d_f$ and $d_g,$ respectively.
Then $\mathrm{SOL}(f^\infty, g^\infty) = \{0\}$ if, and only if, for any polynomial maps $p, q \colon \mathbb{R}^n \to \mathbb{R}^n$ of degree at most $d_f - 1$ and $d_g - 1$ respectively, the $\mathrm{PCP}(f + p, g + q)$ has a compact solution set.
\end{proposition}
\begin{proof}
{\em Necessity.} By contradiction, suppose that there exist polynomial maps $p, q \colon \mathbb{R}^n \to \mathbb{R}^n$ of degrees at most $d_f - 1$ and $d_g - 1$ respectively such that the set $\mathrm{SOL}(f + p, g + q)$ is unbounded. Then there exists a sequence $\{x^k\} \subset \mathbb{R}^n$ such that $\|x^k\| \to \infty$ as $k \to \infty$ and
$$\min\{f(x^k) + p(x^k), g(x^k) + q(x^k)\} = 0 \quad \textrm{ for all } k.$$
Equivalently,
$$\min \left\{\frac{f(x^k) + p(x^k)}{\|x^k\|^{d_f}}, \frac{g(x^k) + q(x^k)}{\|x^k\|^{d_g}} \right\} = 0 \quad \textrm{ for all } k.$$
Let $k \to \infty$ and assume (without loss of generality) $\lim \frac{x^k}{\|x^k\|} = x.$ Then it is not hard to see that
\begin{eqnarray*}
\frac{f(x^k)}{\|x^k\|^{d_f}} \to f^\infty(x), \quad \frac{p(x^k)}{\|x^k\|^{d_f}} \to  0, \quad 
\frac{g(x^k)}{\|x^k\|^{d_g}} \to g^\infty(x), \quad \frac{q(x^k)}{\|x^k\|^{d_g}} \to 0.
\end{eqnarray*}
Therefore, $\min\{f^\infty(x), g^\infty(x)\} = 0,$ or equivalently, $x \in \mathrm{SOL}(f^\infty, g^\infty).$ By the assumption, then $x = 0.$ As $\|x\| = 1,$ we reach a contradiction.

{\em Sufficiency.} By definition, $p := f^\infty - f$ and $q := g^\infty - g$ are polynomial maps of degree at most $d_f - 1$ and $d_g - 1,$ respectively. By the assumption, 
the $\mathrm{PCP}(f^\infty, g^\infty) = \mathrm{PCP}(f + p, g + q)$ has a compact solution set. This, in turn, implies that $\mathrm{SOL}(f^\infty, g^\infty) = \{0\}$ because the polynomial maps $f^\infty$ and $g^\infty$ are homogeneous.
\end{proof}

\begin{remark}{\rm
As the leading terms $f^\infty$ and $g^\infty$ are homogeneous, $\mathrm{SOL}(f^\infty, g^\infty)$ contains $0$ and is invariant under multiplication by positive numbers. Moreover, it is clear that
\begin{eqnarray*}
\mathrm{SOL}(f^\infty, g^\infty) = \{0\} \quad & \Longleftrightarrow & \quad [\min\{f^\infty(x), g^\infty(x)\} = 0  \Leftrightarrow x = 0].
\end{eqnarray*}
In particular, when $f$ is the identity map and $g$ is an affine map (i.e., $g$ is a polynomial map of degree $1$), the condition that $\mathrm{SOL}(f^\infty, g^\infty) = \{0\}$ means that the matrix associated to the linear map $g^\infty$ is an {\em $R_0$-matrix} (see \cite{Cottle2009}).
}\end{remark}

\begin{proposition} \label{MD4}
Let $f, g \colon \mathbb{R}^n \to \mathbb{R}^n$ be polynomial maps of positive degrees $d_f$ and $d_g,$ respectively, 
and let $\frak{m}^\infty(x) := \min\{f^\infty(x), g^\infty(x)\}.$ Suppose the following conditions hold:
\begin{itemize}
\item[(i)] $\frak{m}^\infty(x) = 0 \Rightarrow x = 0;$ and

\item[(ii)] $\deg (\frak{m}^\infty, 0) \ne 0.$
\end{itemize}
Then, for any polynomial maps $p, q \colon \mathbb{R}^n \to \mathbb{R}^n$ of degree at most $d_f - 1$ and $d_g - 1$ respectively, 
the $\mathrm{PCP}(f + p, g + q)$ has a nonempty compact solution set.
 \end{proposition}
\begin{proof}
Let $p, q \colon \mathbb{R}^n \to \mathbb{R}^n$ be polynomial maps of degree at most $d_f - 1$ and $d_g - 1$ respectively. Consider the homotopy
$$H(x, t) := \min \left\{(1 - t) f^\infty(x) + t (f(x) + p(x)), (1 - t) g^\infty(x) + t (g(x) + q(x)) \right\},$$
where $(x, t) \in \mathbb{R}^n \times [0, 1].$ Then 
$$H(x, 0) = \frak{m}^\infty(x) \quad \textrm{ and } \quad H(x, 1) = \min\{f(x) + p(x), g(x) + q(x)\}.$$
Since the condition (i) holds, a standard argument (as in the proof of Proposition~\ref{MD2}) shows that the set
\begin{eqnarray*}
\{x \in \mathbb{R}^n \ : \ H(x, t) = 0 \quad \textrm{ for some } \quad t \in [0, 1]\}
\end{eqnarray*}
is bounded, hence contained in some bounded open set $\Omega$ in $\mathbb{R}^n.$ Then, by the homotopy invariance property of degree, we have
\begin{eqnarray*}
\deg (H(\cdot, 1), \Omega, 0) \ = \ \deg (H(\cdot, 0), \Omega, 0) &=& \deg (\frak{m}^\infty, 0) \ \ne \ 0.
\end{eqnarray*}
So, $H(\cdot, 1)$ has a zero in $\Omega.$ This proves that the $\mathrm{PCP}(f + p, g + q)$ has a solution.
Finally, the compactness of the solution set $\mathrm{SOL}(f + p, g + q)$ follows immediately from Proposition~\ref{MD2}.
\end{proof}

\begin{proposition} \label{MD8}
Let $f, g \colon \mathbb{R}^n \to \mathbb{R}^n$ be polynomial maps of positive degrees $d_f$ and $d_g,$ respectively, 
and let $\frak{m}^\infty(x) := \min\{f^\infty(x), g^\infty(x)\}.$ Suppose there exists a vector $x^{\mathrm{ref}} \in \mathbb{R}^n$ such that the set
$$\{x \in \mathbb{R}^n \ : \ \langle x - x^{\mathrm{ref}}, \frak{m}^\infty(x) \rangle \le 0\}$$
is bounded. Then, for any polynomial maps $p, q \colon \mathbb{R}^n \to \mathbb{R}^n$ of degrees at most $d_f - 1$ and $d_g - 1$ respectively, 
the $\mathrm{PCP}(f + p, g + q)$ has a nonempty compact solution set.
\end{proposition}

\begin{proof}
It suffices to show that all the assumptions of Proposition~\ref{MD4} are satisfied. 
Indeed, note, by assumption, that the set $\{x \in \mathbb{R}^n \ : \ \frak{m}^\infty(x) = 0\}$ is compact. Since the polynomial maps $f^\infty$ and $g^\infty$ are homogeneous, this implies that
$$\frak{m}^\infty(x) = 0 \quad \Rightarrow \quad x = 0.$$
We next prove that $\deg (\frak{m}^\infty, 0) = 1.$ To see this, consider the homotopy
$$H(x, t) := (1 - t) (x - x^{\mathrm{ref}}) + t \frak{m}^\infty(x), \quad (x, t) \in \mathbb{R}^n \times [0, 1].$$
Then $H(x, 0) = x - x^{\mathrm{ref}}$ and $H(x, 1) = \frak{m}^\infty(x).$

Let
\begin{eqnarray*}
X &:=& \{x \in \mathbb{R}^n \ : \ H(x, t) = 0 \quad \textrm{ for some } \quad t \in [0, 1]\}.
\end{eqnarray*}
We have $X$ is bounded. In fact, if it is not the case, then there exist sequences $\{x^k\} \subset \mathbb{R}^n$ with $\lim_{k \to \infty} \|x^k\| = +\infty$ and $\{t^k\} \subset [0, 1]$ such that $H(x^k, t^k) = 0$ for all $k.$ Then
\begin{eqnarray*}
(1 - t^k) \|x^k - x^{\mathrm{ref}}\|^2 + t^k \langle x^k - x^{\mathrm{ref}}, \frak{m}^\infty(x^k) \rangle \ = \ \langle x^k - x^{\mathrm{ref}}, H(x^k, t^k) \rangle &=& 0,
\end{eqnarray*}
which is impossible since $\lim_{k \to \infty} \|x^k - x^{\mathrm{ref}}\| = +\infty$ and $\langle x^k - x^{\mathrm{ref}}, \frak{m}^\infty(x^k) \rangle > 0$ for all $k$ sufficiently large.

Therefore, the set $X$ is contained in some bounded open set $\Omega$ in $\mathbb{R}^n.$ Since  $H(x, 0) = x - x^{\mathrm{ref}}$ and $x^{\mathrm{ref}} \in \Omega,$ it follows that  $\deg (H(\cdot, 0), \Omega, 0)$ is well defined and equal to unity. By the homotopy invariance property of degree, then
\begin{eqnarray*}
\deg (\frak{m}^\infty, 0) \ = \ \deg (H(\cdot, 1), \Omega, 0) \ = \ \deg (H(\cdot, 0), \Omega, 0) &=& 1.
\end{eqnarray*}
From Proposition~\ref{MD4}, we get the stated conclusion.
\end{proof}

\begin{remark}{\rm
Let $f, g \colon \mathbb{R}^n \to \mathbb{R}^n$ be polynomial maps of positive degrees $d_f$ and $d_g,$ respectively, and let $\frak{m}^\infty(x) := \min\{f^\infty(x), g^\infty(x)\}.$ Then it is not hard to see that the following conditions are equivalent:
\begin{itemize}
\item[(i)] $\frak{m}^\infty(x) = 0 \Rightarrow x = 0.$

\item[(ii)] There exists a constant $c > 0$ such that
\begin{eqnarray*}
\|\frak{m}^\infty(x)\|  & \ge & c \|x\|^{\min\{d_f, d_g\}} \quad \textrm{ for } \quad \|x\| \ge 1.
\end{eqnarray*}

\item[(iii)] There exists a constant $c > 0$ such that
\begin{eqnarray*}
\|\frak{m}^\infty(x) \| & \ge & c \|x\|^{\max\{d_f, d_g\}}  \quad \textrm{ for } \quad \|x\| \le 1.
\end{eqnarray*}
\end{itemize}
In particular, if one of the above equivalent conditions is satisfied, then for all $y \in \mathbb{R}^n$ with $y \ge 0,$ the set 
\begin{eqnarray*}
\{x \in \mathbb{R}^n \ : \ 0 \le \frak{m}(x) \le y\}
\end{eqnarray*}
is compact. As we shall not use these facts, we leave the proof as an exercise.
}\end{remark}

The next proposition is inspired by the results in \cite{Karamardian1976} and \cite{Gowda2017}.

\begin{proposition}\label{MD7}
Let $f, g \colon \mathbb{R}^n \to \mathbb{R}^n$ be polynomial maps of positive degrees $d_f$ and $d_g,$ respectively. Suppose the following two conditions hold:
\begin{itemize}
\item[(i)] $\mathrm{SOL}(f^\infty, g^\infty) = \{0\} = \mathrm{SOL}(f^\infty, g^\infty + d)$ for some vector $d > 0$ in $\mathbb{R}^n;$ and

\item[(ii)] The local (topological) degree of the polynomial map $f^\infty \colon \mathbb{R}^n \to \mathbb{R}^n$ at $0$ is well-defined and nonzero.
\end{itemize}
Then, for any polynomial maps $p, q \colon \mathbb{R}^n \to \mathbb{R}^n$ of degree at most $d_f - 1$ and $d_g - 1$ respectively, 
the $\mathrm{PCP}(f + p, g + q)$ has a nonempty compact solution set.
 \end{proposition}
 
\begin{proof}
By Proposition~\ref{MD4}, it suffices to show that the local (topological) degree of the (continuous) map 
$$\frak{m}^\infty \colon \mathbb{R}^n \to \mathbb{R}^n, \quad x \mapsto \min\{f^\infty(x), g^\infty(x)\},$$ at $0$ is well-defined and nonzero.
To see this, consider the homotopy
$$H(x, t) := \min \left\{f^\infty(x), (1 - t) g^\infty(x) + t (g^\infty(x) + d) \right\},$$
where $(x, t) \in \mathbb{R}^n \times [0, 1].$ Then 
$$H(x, 0) = \frak{m}^\infty(x) \quad \textrm{ and } \quad H(x, 1) = \min\{f^\infty(x), g^\infty(x) + d\}.$$
Since $\mathrm{SOL}(f^\infty, g^\infty) = \{0\},$ by a normalization argument (as in the proof of Proposition~\ref{MD2}), we see that the zero set
\begin{eqnarray*}
\{x \in \mathbb{R}^n \ : \ H(x, t) = 0 \quad \textrm{ for some } \quad t \in [0, 1]\}
\end{eqnarray*}
is bounded, hence contained in some bounded open set $\Omega$ in $\mathbb{R}^n.$ In particular, we get
$$\frak{m}^\infty(x) = 0 \quad \Rightarrow \quad x = 0.$$
Furthermore, by the homotopy invariance property of degree, we have
\begin{eqnarray*}
\deg (\frak{m}^\infty, 0) \ = \ \deg (H(\cdot, 0), \Omega, 0) &=& \deg (H(\cdot, 1), \Omega, 0).
\end{eqnarray*}

On the other hand, it is clear that when $x$ is close to zero, $f^\infty(x)$ and $g^\infty(x) + d$ are, respectively, close to $0$ and $g^\infty(0) + d = d > 0.$ Hence, for all $x$ close to zero, 
\begin{eqnarray*}
H(x, 1) &=& \min\{f^\infty(x), g^\infty(x) + d\} \ = \ f^\infty(x).
\end{eqnarray*}
This, together with the condition~(ii), implies that $\deg (H(\cdot, 1), \Omega, 0) \ne 0.$ Therefore,
\begin{eqnarray*}
\deg (\frak{m}^\infty, 0) &=& \deg (H(\cdot, 0), \Omega, 0) \ \ne \ 0.
\end{eqnarray*}
From Proposition~\ref{MD4}, we get the stated conclusion.
\end{proof}

The following result, which is inspired by \cite[Proposition~2.2.3]{Facchinei2003}, gives a sufficient condition under which the PCP$(f, g)$ has a nonempty compact solution set.
\begin{proposition}\label{MD5}
Let $f, g \colon \mathbb{R}^n \to \mathbb{R}^n$ be polynomial maps. If there exists a vector $x^{\mathrm{ref}} \in \mathbb{R}^n$ such that the set
$$\{x \in \mathbb{R}^n \ : \ \langle x - x^{\mathrm{ref}}, \frak{m}(x) \rangle \le 0\}$$
is bounded, then the $\mathrm{PCP}(f, g)$ has a nonempty compact solution set.
\end{proposition}
\begin{proof}
Consider the homotopy
$$H(x, t) := (1 - t) (x - x^{\mathrm{ref}}) + t \frak{m}(x), \quad (x, t) \in \mathbb{R}^n \times [0, 1].$$
Then $H(x, 0) = x - x^{\mathrm{ref}}$ and $H(x, 1) = \frak{m}(x).$

We first show that the set 
\begin{eqnarray*}
X &:=& \{x \in \mathbb{R}^n \ : \ H(x, t) = 0 \quad \textrm{ for some } \quad t \in [0, 1]\}
\end{eqnarray*}
is bounded. Indeed, if it is not the case, then there exist sequences $\{x^k\} \subset \mathbb{R}^n$ with $\lim_{k \to \infty} \|x^k\| = +\infty$ and $\{t^k\} \subset [0, 1]$ such that $H(x^k, t^k) = 0$ for all $k.$ Then we have
\begin{eqnarray*}
(1 - t^k) \|x^k - x^{\mathrm{ref}}\|^2 + t^k \langle x^k - x^{\mathrm{ref}}, \frak{m}(x^k) \rangle \ = \ \langle x^k - x^{\mathrm{ref}}, H(x^k, t^k) \rangle &=& 0,
\end{eqnarray*}
which contradicts to the facts that $\lim_{k \to \infty} \|x^k - x^{\mathrm{ref}}\| = +\infty$ and $\langle x^k - x^{\mathrm{ref}}, \frak{m}(x^k) \rangle > 0$ for all $k$ sufficiently large.

Therefore, the set $X$ is contained in some bounded open set $\Omega$ in $\mathbb{R}^n.$ Since  $H(x, 0) = x - x^{\mathrm{ref}}$ and $x^{\mathrm{ref}} \in \Omega,$ it follows that  $\deg (H(\cdot, 0), \Omega, 0)$ is well defined and equal to unity. By the homotopy invariance property of degree, then
\begin{eqnarray*}
\deg (H(\cdot, 1), \Omega, 0) \ = \ \deg (H(\cdot, 0), \Omega, 0) &=& 1.
\end{eqnarray*}
So, $H(\cdot, 1),$ that is $\frak{m}$ has a zero in $\Omega.$ This proves that the $\mathrm{PCP}(f, g)$ has a solution.
Finally, the compactness of the solution set $\mathrm{SOL}(f, g)$ follows immediately from our assumption.
\end{proof}

The following notion generalizes the well-known notion of $P$-functions \cite{Cottle2009, Facchinei2003, More1973}.

\begin{definition}{\rm
A pair map $(f, g) \colon K \subset \mathbb{R}^n \to \mathbb{R}^n \times \mathbb{R}^n$ is said to be a {\em $P$-function} if for every $x, y$ in $K$ with $x \ne y,$ there exists an index $i$ such that
\begin{eqnarray*}
\big(f_i(x)  - f_i(y)\big) \big(g_i(x) - g_i(y) \big) &>& 0.
\end{eqnarray*}
}\end{definition}

\begin{proposition} \label{MD6}
Let $f, g \colon \mathbb{R}^n \to \mathbb{R}^n$ be polynomial maps. If the restriction of the map $(f, g)$ on the set
$$K := \{x \in \mathbb{R}^n \ : \ f(x) \ge 0, g(x) \ge 0\}$$ 
is a $P$-function, then the $\mathrm{PCP}(f, g)$ has at most one solution.
\end{proposition}
\begin{proof}
Suppose that the pair map $(f, g)$ is $P$-function on $K.$ If $x \ne x'$ are two distinct solutions of the $\mathrm{PCP}(f, g),$  we have for all $i = 1, \ldots, n,$ that
\begin{eqnarray*}
\big(f_i(x)  - f_i(x')\big) \big(g_i(x) - g_i(x') \big)
&=& f_i(x) g_i(x) + f_i(x') g_i(x') - f_i(x') g_i(x) - f_i(x) g_i(x') \\
&= & - f_i(x') g_i(x) - f_i(x) g_i(x') \ \le \ 0,
\end{eqnarray*}
which contradicts our assumption.
\end{proof}

The example below shows that even for a $P$-function, the corresponding complementarity problem may have no solution.
\begin{example}{\rm
Consider the problem $\mathrm{PCP}(f, g),$ where
$$f(x,y) = g(x, y) := (x, xy - 1) \quad \textrm{ for } \quad (x, y) \in \mathbb{R}^2.$$ 
It is easily seen that the restriction of $(f, g)$ on the set
\begin{eqnarray*}
\{(x, y) \in \mathbb{R}^2 \ : \ x \ge 0, xy - 1 \ge 0\} 
\end{eqnarray*}
is a $P$-function. Nevertheless the $\mathrm{PCP}(f, g)$ has no solution.
}\end{example}

On the other hand, it is clear that the two component maps of a $P$-function must be injective. This observation leads to the next result.

\begin{proposition}
Let $f, g \colon \mathbb{R}^n \to \mathbb{R}^n$ be polynomial maps. 
The $\mathrm{PCP}(f, g)$ has a nonempty compact solution set under either one of the following two
conditions:
\begin{itemize}
\item[(i)] the map $f$ is injective and the set
\begin{eqnarray*}
\{x \in \mathbb{R}^n \ : \ f(x) \ge 0 \textrm{ and } \langle f(x), g(x) \rangle \le 0\}
\end{eqnarray*}
is bounded;

\item[(ii)] the map $g$ is injective and the set
\begin{eqnarray*}
\{x \in \mathbb{R}^n \ : \ g(x) \ge 0 \textrm{ and } \langle f(x), g(x) \rangle \le 0\}
\end{eqnarray*}
is bounded.
\end{itemize}
\end{proposition}
\begin{proof}
Without loss of generality, assume (i) holds. Since the map $f$ is injective, it follows from \cite{Bailynicki-Birula1962} that $f$ is surjective. Consequently, we can see that $f$ is a homeomorphism from $\mathbb{R}^n$ into itself. 

Consider the homotopy
$$H(x, t) := \min \left\{f(x), (1 - t) f(x) + t g(x)) \right\},$$
where $(x, t) \in \mathbb{R}^n \times [0, 1].$ Observe that 
$$H(x, 0) = f(x) \quad \textrm{ and } \quad H(x, 1) = \min\{f(x), g(x)\}.$$
We first show that the set 
\begin{eqnarray*}
X &:=& \{x \in \mathbb{R}^n \ : \ H(x, t) = 0 \quad \textrm{ for some } \quad t \in [0, 1]\}
\end{eqnarray*}
is bounded. Indeed, if it is not the case, then there exist sequences $\{x^k\} \subset \mathbb{R}^n$ with $\lim_{k \to \infty} \|x^k\| = +\infty$ and $\{t^k\} \subset \mathbb{R}$ such that $H(x^k, t^k) = 0$ for all $k.$ Clearly, the following facts hold:
\begin{eqnarray*}
f(x^k) \ge 0,  \quad (1 - t^k) f(x^k) + t^k g(x^k) \ge 0, \quad \langle f(x^k), (1 - t^k) f(x^k) + t^k g(x^k) \rangle = 0.
\end{eqnarray*}
Since the map $f$ is homeomorphism, $f(x^k) \ne 0$ for all $k$ sufficiently large. On the other hand, the assumption that the set
\begin{eqnarray*}
\{x \in \mathbb{R}^n \ : \ f(x) \ge 0 \textrm{ and } \langle f(x), g(x) \rangle \le 0\}
\end{eqnarray*}
is bounded implies that $\langle f(x^k) , g(x^k) \rangle > 0$ for all $k$ sufficiently large. 

Therefore, for all $k$ sufficiently large, we have
\begin{eqnarray*}
0 \ = \ \langle f(x^k), (1 - t^k) f(x^k) + t^k g(x^k) \rangle  & = & (1 - t^k) \|f(x^k)\|^2 + t^k \langle f(x^k) , g(x^k) \rangle \\
&\ge & \min\{\|f(x^k)\|^2, \langle f(x^k) , g(x^k) \rangle\} \ > \ 0,
\end{eqnarray*}
which is impossible.

Therefore, the set $X$ is contained in some bounded open set $\Omega$ in $\mathbb{R}^n.$ Since $f$ is a homeomorphism, there is a unique $x^* \in \mathbb{R}^n$ such that $f(x^*) = 0,$ and then $\deg (f, x^*)$ is equal to $1$ or $-1$. By the homotopy invariance property of degree, we get
\begin{eqnarray*}
\deg (H(\cdot, 1), \Omega, 0) \ = \ \deg (H(\cdot, 0), \Omega, 0) &=& \deg (f, x^*) \  = \ \pm 1.
\end{eqnarray*}
So, $H(\cdot, 1),$ that is $\min\{f(x), g(x)\}$ has a zero in $\Omega.$ This proves that the $\mathrm{PCP}(f, g)$ has a solution.
Finally, the compactness of the solution set $\mathrm{SOL}(f, g)$ follows immediately from our assumption.
\end{proof}

The next two propositions may be considered generalizations of \cite[Theorems~3.2~and~3.3]{Karamardian1972}.

\begin{proposition}\label{Proposition410}
Let $C \subset \mathbb{R}^n$ be a nonempty compact set such that for every $x \in \mathbb{R}^n \setminus C$ there exists a $y \in C$ satisfying $\langle x - y, \frak{m}(x)\rangle >0.$ Then, the $\mathrm{PCP}(f, g)$ has a nonempty compact solution set.
\end{proposition}
\begin{proof}
For each $u \in \mathbb{R}^n$ let
\begin{eqnarray*}
D_u &:=& \{x \in C \ : \ \langle u-x, \frak{m}(x) \rangle \geq 0\}.
\end{eqnarray*}
It is clear that $D_u$ is compact. Next, we will prove that the intersection of any finite of the $D_{u}$'s is nonempty, i.e, for arbitrary $u^i \in \mathbb{R}^n, i = 1, \ldots, m$ we have $\cap_{i=1}^m D_{u^i} \ne \emptyset.$
	 
To see this, let $D$ be the convex hull of $C \cup \{u^1, \ldots, u^m\}.$ Obviously $D$ is a nonempty compact convex subset in $\mathbb{R}^n.$ Hence, it follows from \cite[Theorem 2.1]{Karamardian1972} that there exists $\bar{x} \in D$ such that 
\begin{eqnarray}\label{ppt4}
\langle x - \bar{x}, \frak{m}(\bar{x})\rangle &\ge& 0 \quad \text{ for all } \quad x \in  D.
\end{eqnarray}
In particular, $\langle u^i - \bar{x}, \frak{m}(\bar{x})\rangle \geq 0$ for $i = 1, \ldots, m.$	 

If $\bar{x} \notin C$ then it follows from our assumption that there exists $y \in C$ such that
\begin{eqnarray*}
\langle \bar{x} - y, \frak{m}(\bar{x}) \rangle &>& 0,
\end{eqnarray*} 
which contradicts \eqref{ppt4}. Hence $\bar{x} \in C$ and so $\bar{x} \in D_{u^i}$ for all $i = 1, \ldots, m.$ From the finite intersection property of compact sets we have $\cap_{u \in \mathbb{R}^n} D_{u} \ne \emptyset,$ which yields the existence of a point $x^* \in C$ satisfying the condition
\begin{eqnarray*}
\langle u- x^*, \frak{m}(x^*) \rangle &\ge&  0 \quad \text{ for all } \quad u \in \mathbb{R}^n.
\end{eqnarray*}
This implies easily that $\frak{m}(x^*) = 0$ and so $x^* \in \mathrm{Sol}(f, g).$ 

Finally, it follows easily from the assumption that the solution set $\mathrm{Sol}(f, g)$ is contained in the set $C$ and so it is a compact.
\end{proof}	 

\begin{corollary}
If there exists a real number $c >0$ such that 
\begin{eqnarray*}
\langle x, \frak{m}(x)-\frak{m}(0) \rangle \geq c \| x \|^2 \quad \textrm{ for } \quad \|x\| > \frac{\| \frak{m}(0) \|}{c},
\end{eqnarray*}
then the $\mathrm{PCP}(f, g)$ has a nonempty compact solution set.
\end{corollary}
\begin{proof}
Let
\begin{eqnarray*}
C &:=& \left\{x \in \mathbb{R}^n \ : \  \| x \| \leq \frac {\| \frak{m}(0) \|}{c}\right\}.
\end{eqnarray*}
Clearly $C$ is compact and contains $0$. Take any $x \notin C,$ i.e.,
\begin{eqnarray*}
c \| x \| &>& \| \frak{m}(0) \|. 
\end{eqnarray*}
This, together with Schwartz's inequality, implies that
\begin{eqnarray*}	
c \| x \|^2  & > & \| x \| \cdot \| \frak{m}(0) \|   \ \geq \ - \langle x, \frak{m}(0) \rangle.
 \end{eqnarray*}
By assumption, then
\begin{eqnarray*}
\langle x, \frak{m}(x) \rangle & \geq & \langle x, \frak{m}(0) \rangle + c  \| x \|^2 \\
&>& \langle x, \frak{m}(0) \rangle -\langle x, \frak{m}(0)  \rangle \ = \ 0.
\end{eqnarray*}
Therefore,
\begin{eqnarray*}
\langle x - 0, \frak{m}(x) \rangle &=& \langle x, \frak{m}(x) \rangle \ > \ 0. 
\end{eqnarray*}
From Proposition~\ref{Proposition410}, we get the desired conclusion.
\end{proof}

\begin{proposition} 
Let $C$ be a nonempty, compact and convex subset in $\mathbb{R}^n$ such that the origin $0$ belongs to the interior of $C$ and that $\langle x, \frak{m}(x) \rangle \geq 0$ for all $x \in \partial C$-the boundary of $C.$ Then, the $\mathrm{PCP}(f, g)$ has a solution.
\end{proposition}
	
\begin{proof}
By \cite[Theorem 2.1]{Karamardian1972}, there exists $\bar{x} \in C$ such that 
\begin{eqnarray}\label{ppt1}
\langle x - \bar{x}, \frak{m}(\bar{x}) \rangle & \geq & 0 \quad \textrm{ for all } \quad x \in C.
\end{eqnarray}		
Since $C$ contains the origin we also have
\begin{eqnarray}\label{ppt2}
\langle \bar{x}, \frak{m}(\bar{x})\rangle &\leq& 0. 
\end{eqnarray}	
We consider two cases:
	
\subsubsection*{Case 1: $\bar{x} \in \partial C$} It follows from the assumption that 
\begin{eqnarray*}
\langle \bar{x}, \frak{m}(\bar{x})\rangle & \geq & 0. 
\end{eqnarray*}
Combining this with \eqref{ppt2} we obtain $\langle \bar{x}, \frak{m}(\bar{x})\rangle = 0.$

On the other hand, for every $i=1, \ldots, n,$ there exist scalars $\alpha_i > 0$ and $\beta_i < 0$ such that $\alpha_i e^i \in \partial C$ and $\beta_i e^i \in \partial C,$
here $e^i$ is the $i$th unit vector in $\mathbb{R}^n.$ Substituting $x = \alpha_i e^i $ into \eqref{ppt1} we obtain
\begin{eqnarray*}
\langle\alpha_i e^i, \frak{m}(\bar{x}) \rangle & \geq & \langle \bar{x}, \frak{m}(\bar{x})\rangle  \ =\  0,
\end{eqnarray*}
which implies that
\begin{eqnarray*}
\alpha_i \cdot \min \{f_i(\bar{x}), g_i(\bar{x})\} &\ge& 0.
\end{eqnarray*}
Consequently, $\min \{f_i(\bar{x}), g_i(\bar{x}) \}\geq 0$ because $\alpha_i$ is positive. Therefore, 
\begin{eqnarray*}
\frak{m}(\bar{x})  &\ge&  0.
\end{eqnarray*}
Similarly, since $\beta_i$ is negative, we also have
\begin{eqnarray*}
\frak{m}(\bar{x}) &\le&  0.
\end{eqnarray*}
Hence $\frak{m}(\bar{x}) = 0$ and so $\bar{x} \in \mathrm{Sol}(f, g).$ 
	
\subsubsection*{Case 2: $\bar{x} \notin \partial C$} Then for all $i = 1, \ldots, n$ and for all $|t|$ small enough, we have $\bar{x} + te^i \in C,$ which together with \eqref{ppt1} gives 
\begin{eqnarray*}
\langle te^i, \frak{m}(\bar{x}) \rangle &\ge& 0.
\end{eqnarray*}
Clearly, this implies that $\frak{m}(\bar{x}) = 0$ and so $\bar{x} \in \mathrm{Sol}(f,g).$ The proof is completed.
\end{proof}

\section{Error bounds} \label{Section5}

In this section, we establish some error bound results for the solution set of polynomial complementarity problems in terms of the natural map with explicit exponents.

Recall that, given polynomial maps $f, g \colon \mathbb{R}^n \to \mathbb{R}^n$ of degree at most $d \ge 1,$ the solution set $\mathrm{SOL}(f, g)$ of the PCP$(f, g)$ is the set of
vectors $x \in \mathbb{R}^n$ satisfying the following constraints
\begin{equation*}
f(x) \ \ge \ 0, \quad g(x) \ \ge \  0, \quad \textrm{ and } \quad \langle f(x), g(x) \rangle \ = \ 0.
\end{equation*}
Clearly, this is a polynomial system with one equality and $2n$ inequalities and with the maximum degree $2d.$ By Lemma~\ref{Lemma21}, for any compact set $K \subset \mathbb{R}^n,$ we may find a constant $c > 0$ satisfying the H\"olderian error bound
\begin{eqnarray*}
c\, \mathrm{dist}(x, \mathrm{SOL}(f, g))^{\alpha} & \le & \sum_{i = 1}^n \left([-f_i(x)]_+ + [-g_i(x)]_+ \right) + | \langle f(x), g(x) \rangle | \quad \textrm{ for all } \quad x \in K,
\end{eqnarray*}
where $\alpha := \mathscr{R}(3n, 2d + 1).$ On the other hand, using the natural map $\frak{m},$ we can improve this error bound, and also strengthen and generalize the recent result of Hu, Wang and Huang \cite{Hu2018-2}.

\begin{theorem} \label{Theorem11}
Let $f, g \colon \mathbb{R}^n \to \mathbb{R}^n$ be polynomial maps of degree at most $d \ge 1.$ For any compact set $K \subset \mathbb{R}^n,$ there exists a constant $c > 0$ such that
\begin{eqnarray}\label{PT2}
c\, \mathrm{dist}(x, \mathrm{SOL}(f, g))^{\alpha} & \le & \|\frak{m}(x)\| \quad \textrm{ for all } \quad x \in K, \label{PT21}
\end{eqnarray}
where $\alpha := \mathscr{R}(3n - 1, d + 1).$
\end{theorem}

The proof of this theorem is done in several steps. We start with the following simple lemma.

\begin{lemma} \label{Lemma11}
For any real numbers $a, b,$ the following inequality holds
\begin{eqnarray*}
\min\{|a| + [-b]_+, [-a]_+ + |b|\} & \le & 2 | \min\{a, b\} |.
\end{eqnarray*}
\end{lemma}
\begin{proof}
By interchange of $a$ and $b,$ we may assume that
\begin{eqnarray*}
|a| + [-b]_+ &\le& [-a]_+ + |b|.
\end{eqnarray*}
There are two cases to be considered.

\subsubsection*{Case 1: $a > 0$}
In this case, we have from the above inequality that
\begin{eqnarray*}
0 \ < \ a + [-b]_+ \ = \ |a| + [-b]_+ &\le& [-a]_+ + |b| \ = \ |b|.
\end{eqnarray*}
Then it is easy to see that $a \le b.$ Consequently, we have
\begin{eqnarray*}
|a| + [-b]_+ &=& a \ = \ |\min\{a, b\}| \ < \ 2 |\min\{a, b\}|.
\end{eqnarray*}

\subsubsection*{Case 2: $a \le 0$}
If $b \ge 0,$ then
\begin{eqnarray*}
|a| + [-b]_+ &=& -a \ = - \min\{a, b\} \ = \ |\min\{a, b\}| \ \le \ 2 |\min\{a, b\}|.
\end{eqnarray*}
Finally, assume that $b < 0.$ In this case, we have
\begin{eqnarray*}
|a| + [-b]_+ &=& -a - b \le - 2\min\{a, b\} \ = \ 2 |\min\{a, b\}|.
\end{eqnarray*}
The lemma is proved.
\end{proof}

For each (possibly empty) set $I \subset \{1, \ldots, n\},$ we define the function $\Phi_I \colon \mathbb{R}^n \to \mathbb{R}$ by
\begin{eqnarray*}
\Phi_I(x) &:=& \sum_{i \in I} (|f_i(x)| + [-g_i(x)]_+) + \sum_{i \not \in I} ([-f_i(x)]_+ + |g_i(x)|).
\end{eqnarray*}

\begin{lemma} \label{BD2}
For all $x \in \mathbb{R}^n,$ the following inequality holds
\begin{eqnarray*}
\min_I \Phi(x) & \le & 2 \sqrt{n} \|\frak{m}(x)\|.
\end{eqnarray*}
\end{lemma}

\begin{proof}
Take any $x \in \mathbb{R}^n$ and let $I' \subset \{1, \ldots, n\}$ be an index set such that 
$$\Phi_{I'}(x) := \min_I \Phi_I(x).$$ 
We have
\begin{eqnarray*}
|f_i(x)| + [-g_i(x)]_+ &\le& [-f_i(x)]_+ + |g_i(x)|  \quad \textrm{ for all } \quad i \in I'.
\end{eqnarray*}
Indeed, if this fails to hold at some index $i \in I',$ then, with $I'' := I' \setminus \{i\},$ we would have 
\begin{eqnarray*}
\min_I \Phi_I(x) \ = \ \Phi_{I'}(x) &>& \Phi_{I''}(x) \ \ge \ \min_I \Phi_I(x),
\end{eqnarray*}
which is a contradiction. Similarly, we also have
\begin{eqnarray*}
[-f_i(x)]_+ + |g_i(x)| &\le& |f_i(x)| + [-g_i(x)]_+  \quad \textrm{ for all } \quad i \not \in I'.
\end{eqnarray*}

By Lemma~\ref{Lemma11}, therefore
\begin{eqnarray*}
\Phi_{I'}(x) 
&=&  \sum_{i \in I'} (|f_i(x)| + [-g_i(x)]_+) + \sum_{i \not \in I'} ([-f_i(x)]_+ + |g_i(x)|) \\
&\le& \sum_{i \in I'} 2 |\min\{f_i(x), g_i(x)\}| +  \sum_{i \not \in I'} 2 |\min\{f_i(x), g_i(x)\}|  \\
&=& 2 \sum_{i = 1}^n |\min\{f_i(x), g_i(x)\}|  \\
&\le& 2 \sqrt{n} \|\frak{m}(x)\|,
\end{eqnarray*}
which proves the desired inequality.
\end{proof}

\begin{remark}{\rm
Analysis similar to that in the proof of Lemma~\ref{BD2} shows that 
\begin{eqnarray*}
\|\frak{m}(x)\|  &\le& \min_I \Phi_I(x) \quad \textrm{ for all } \quad x \in \mathbb{R}^n.
\end{eqnarray*}
As we shall not use this inequality, we leave the proof as an exercise.
}\end{remark}

The next lemma is an intermediate step toward the desired error bound in Theorem~\ref{Theorem11}.

\begin{lemma} \label{BD3}
For any compact set $K \subset \mathbb{R}^n,$ there exists a constant $c > 0$ such that
\begin{eqnarray*}
c\, \mathrm{dist}(x, \mathrm{SOL}(f, g))^{\alpha} & \le & \min_I \Phi(x) \quad \textrm{ for all } \quad x \in K,
\end{eqnarray*}
where $\alpha := \mathscr{R}(3n - 1, d + 1).$
\end{lemma}

\begin{proof}
We first assume that $\mathrm{SOL}(f, g) = \emptyset.$ By convention, $\mathrm{dist}(x, \mathrm{SOL}(f, g)) = 1;$ furthermore, by definition, we have
$$\min_I \Phi(x) \ne 0 \quad \textrm{ for all } \quad x \in \mathbb{R}^n.$$
Therefore, the desired conclusion holds with the constant $c := \min_{x \in K} \min_I \Phi(x) > 0.$ 

Now assume that the solution set $\mathrm{SOL}(f, g)$ is not empty. Recall that, for each (possibly empty) set $I \subset \{1, \ldots, n\},$ the function $\Phi_I \colon \mathbb{R}^n \to \mathbb{R}^n$ is defined by
$$\Phi_I(x) := \sum_{i \in I} (|f_i(x)| + [-g_i(x)]_+) + \sum_{i \not \in I} ([-f_i(x)]_+ + |g_i(x)|).$$
By definition, $\Phi_I$ is nonnegative on $\mathbb{R}^n$ and, furthermore, a point $x$ belongs to the zero set $\Phi_I^{-1}(0)$ if, and only if, it satisfies the following constraints 
$$f_i(x) = 0, g_i(x) \ge 0 \textrm{ for } i \in I \quad \textrm{ and } \quad  f_i(x) \ge 0, g_i(x)  = 0 \textrm{ for } i \not \in I.$$
Note that this is a polynomial system with $n$ equalities and $n$ inequalities and with the maximum degree $d.$ 
By Lemma~\ref{Lemma21}, we may find a constant $c_I > 0$ satisfying the following error bound
\begin{eqnarray} \label{PT4}
c_I\, \mathrm{dist}(x, \Phi_I^{-1}(0))^{\alpha} & \le & \Phi_I(x) \quad \textrm{ for all } \quad x \in K.
\end{eqnarray}
Note that, since $K$ is a compact set, the error bound \eqref{PT4} holds even when $\Phi_I^{-1}(0)$ is an empty set.

On the other hand, we have $\mathrm{SOL}(f, g) = \cup_{I \in \mathscr{I}} \Phi_I^{-1}(0),$ where $\mathscr{I}$ denotes the family of all subsets $I$ of $\{1, \ldots, n\}$ for which the zero set $\Phi_I^{-1}(0)$ is not empty. Then it is easily seen that 
\begin{eqnarray*}
\mathrm{dist}(x, \mathrm{SOL}(f, g)) &=& \min_{I \in \mathscr{I}} \mathrm{dist}(x, \Phi_I^{-1}(0)) \quad \textrm{ for all } \quad x \in \mathbb{R}^n.
\end{eqnarray*}
Furthermore, since the function $\mathbb{R}^n \rightarrow \mathbb{R}, x \mapsto \mathrm{dist}(x, \mathrm{SOL}(f, g)),$ is continuous and the set $K$ is compact, there exists a constant $M \ge 1$ such that
\begin{eqnarray*}
\mathrm{dist}(x, \mathrm{SOL}(f, g)) &\le& M  \quad \textrm{ for all } \quad x \in K.
\end{eqnarray*}
It follows that
\begin{eqnarray*}
\mathrm{dist}(x, \mathrm{SOL}(f, g)) &\le& M\, \min_{I \not \in \mathscr{I}} \mathrm{dist}(x, \Phi_I^{-1}(0)) \quad \textrm{ for all } \quad x \in K,
\end{eqnarray*}
because, by convention, we set $\mathrm{dist}(x,\emptyset) = 1.$ Therefore
\begin{eqnarray*}
\mathrm{dist}(x, \mathrm{SOL}(f, g)) &\le& M \min_{I} \mathrm{dist}(x, \Phi_I^{-1}(0)) \quad \textrm{ for all } \quad x \in K,
\end{eqnarray*}
where the minimum is taken over all subsets $I$ of $\{1, \ldots, n\}.$ Letting $c := \min_I \frac{c_I}{M^\alpha} > 0,$ we get for all $x \in K,$
\begin{eqnarray*}
c\, \mathrm{dist}(x, \mathrm{SOL}(f, g))^\alpha 
&\le& c\, M^\alpha \min_{I} \mathrm{dist}(x, \Phi_I^{-1}(0))^\alpha \\
&=& M^\alpha \min_{I} c\, \mathrm{dist}(x, \Phi_I^{-1}(0))^\alpha \\
&\le& M^\alpha \min_{I} \frac{c_I}{M^\alpha} \, \mathrm{dist}(x, \Phi_I^{-1}(0))^\alpha \\
&=& \min_{I} {c_I} \, \mathrm{dist}(x, \Phi_I^{-1}(0))^\alpha \\
&\le& \min_I \Phi_I(x),
\end{eqnarray*}
where the last inequality follows from \eqref{PT4}.
\end{proof}

\begin{proof}[Proof of Theorem~\ref{Theorem11}]
This is an immediate consequence of Lemmas~\ref{BD2} and \ref{BD3}.
\end{proof}

The following example indicates that in general the error bound~\eqref{PT21} cannot hold globally for all  $x \in \mathbb{R}^n.$
\begin{example}{\rm
Consider the problem $\mathrm{PCP}(f, g)$ with
$$f(x,y) = g(x, y) := (y - 1, xy - 1) \quad \textrm{ for } \quad (x, y) \in \mathbb{R}^2.$$ 
It is easily seen that $\mathrm{SOL}(f, g) = \{(1, 1)\}.$ Consider the sequence $z^k := (k, \frac{1}{k})$ for $k \ge 1.$ 
As $k \to +\infty,$ we have 
\begin{eqnarray*}
\frak{m}(z^k) &=& \left(\frac{1}{k} - 1, 0 \right) \ \to \ (-1, 0), \\
\mathrm{dist}(z^k, \mathrm{SOL}(f, g) ) &=& \sqrt{(k - 1)^2 + \left(\frac{1}{k} - 1\right)^2} \ \to \ + \infty.
\end{eqnarray*}
It turns out that there cannot exist any positive scalars $c$ and $\alpha$ such that
\begin{eqnarray*}
c\, \mathrm{dist}(z^k, \mathrm{SOL}(f, g))^{\alpha}  & \le & \|\frak{m}(z^k)\|
\end{eqnarray*}
for all $k$ sufficiently large. Thus, a global error bound with the natural map $\frak{m},$ even raised to any positive power, cannot hold in this case.
}\end{example}

The next result shows that for the $\mathrm{PCP}(f, g),$ where $f, g$ are affine maps, the validity of a Lipschitzian error bound for the solution set $\mathrm{SOL}(f, g)$ over compact sets in terms of the natural map $\mathfrak{m}$ can be completely characterized. This is possible because, in this case, the solution set $\mathrm{SOL}(f, g)$ can be described as the solution set of a finite number of linear equalities and inequalities and so the well-known Hoffman's error bound analysis is applicable. 

\begin{theorem} \label{Theorem12}
Assume $f, g \colon \mathbb{R}^n \to \mathbb{R}^n$ are affine maps. For any compact set $K \subset \mathbb{R}^n,$ there exists a constant $c > 0$ such that
\begin{eqnarray*}
c\, \mathrm{dist}(x, \mathrm{SOL}(f, g)) & \le & \|\frak{m}(x)\| \quad \textrm{ for all } \quad x \in K.
\end{eqnarray*}
\end{theorem}
\begin{proof}
By assumption, for each index set $I \subset \{1, \ldots, n\},$ the zero set $\Phi_I^{-1}(0)$ is given by a system of linear equalities and inequalities. By Hoffman's error bound for polyhedra \cite{Hoffman1952}, if $\Phi_I^{-1}(0)$ is not empty then we can find a constant $c_I > 0$ such that
\begin{eqnarray*}
c_I\, \mathrm{dist}(x, \Phi_I^{-1}(0)) & \le & \Phi_I(x) \quad \textrm{ for all } \quad x \in K.
\end{eqnarray*}
Since $K$ is a compact set, it is easy to see that this fact also holds when $\Phi_I^{-1}(0)$ is an empty set. 
The rest of the proof is similar to the one of Theorem~\ref{Theorem11} given above. The details are left to the reader.
\end{proof}

\begin{remark}{\rm
At this point we would like to mention that Theorem~\ref{Theorem12} can also be deduced from a result of Robinson \cite{Robinson1981} (see also \cite{Luo1992, Luo1994}) on a locally upper Lipschitzian property of polyhedral multifunctions\footnote{A multifunction is {\em polyhedral} if its graph is the union of finitely many polyhedral convex sets.}. More precisely, when $f, g \colon \mathbb{R}^n \to \mathbb{R}^n$ are affine maps, the inverse of the natural map $\frak{m}$ is a polyhedral multifunction and thus, by Robinson's result
\cite[Proposition~1]{Robinson1981}, is locally upper Lipschitzian at the origin, that is, there exist scalars $c > 0$ and $\epsilon > 0$ such that
\begin{eqnarray*}
\frak{m}^{-1}(y) & \subset & \frak{m}^{-1}(0) + c \|y\| \mathbb{B},
\end{eqnarray*}
for all $y \in \mathbb{R}^n$ with $\|y\| \le \epsilon,$ where $\mathbb{B}$ denotes the unit closed ball in $\mathbb{R}^n.$ This statement implies easily Theorem~\ref{Theorem12}.
}\end{remark}

The next result, which is inspired by the works of Gowda \cite{Gowda1996}, Mangasarian and Ren \cite{Mangasarian1994}, and Facchinei and Pang \cite[Theorem~6.3.12]{Facchinei2003}, provides a global H\"olderian error bound in terms of the natural map for a broad class of PCPs.

\begin{theorem} \label{Theorem13}
Let $f, g \colon \mathbb{R}^n \to \mathbb{R}^n$ be polynomial maps of degree at most $d \ge 1.$ If $\mathrm{SOL}(f^\infty, g^\infty) = \{0\}$ then there exists a constant $c > 0$ such that
\begin{eqnarray*}
c\, \min\{\mathrm{dist}(x, \mathrm{SOL}(f, g)), \mathrm{dist}(x, \mathrm{SOL}(f, g))^{\alpha}\} & \le & \|\frak{m}(x)\| \quad \textrm{ for all } \quad x \in \mathbb{R}^n,
\end{eqnarray*}
where we put
$$\alpha := 
\begin{cases}
\mathscr{R}(3n - 1, d + 1) & \textrm{ if } d > 1,\\
1 & \textrm{ otherwise.}
\end{cases}$$
\end{theorem}

In order to prove this theorem we need the following lemma.
\begin{lemma} \label{BD4}
If $\mathrm{SOL}(f^\infty, g^\infty) = \{0\},$ there exist constants $c > 0$ and $R > 0$ such that
\begin{eqnarray*}
\|\frak{m}(x)\| & \ge & c\, \| x \|  \quad \textrm{ for all } \quad \|x\| \ge R.
\end{eqnarray*}
\end{lemma}
\begin{proof}
By contradiction, assume that there exists a sequence $\{x^k\} \subset \mathbb{R}^n$ such that $\|x^k\| \to \infty$ as $k \to \infty$ and
\begin{eqnarray*}
\|\frak{m}(x^k)\| & < & \frac{1}{k} \|x^k\|  \quad \textrm{ for all } \quad k \in \mathbb{N}.
\end{eqnarray*}
Since the number of all subsets of $\{1, \ldots, n\}$ is finite, we can assume that there exists a (possibly empty) set $I \subset \{1, \ldots, n\}$ such that for all $k \in \mathbb{N},$
\begin{eqnarray*}
f_i(x^k) & \le & g_i(x^k) \quad \textrm{ for } i \in I,\\
f_i(x^k) & > &  g_i(x^k) \quad \textrm{ for } i \not \in I.
\end{eqnarray*}
Consequently, we have for all $k \in \mathbb{N},$
\begin{eqnarray*}
|f_i(x^k)| & < & \frac{1}{k} \|x^k\| \quad \textrm{ for } i \in I,\\
|g_i(x^k)| & < & \frac{1}{k} \|x^k\| \quad \textrm{ for } i \not \in I.
\end{eqnarray*}
Let $k \to \infty$ and assume (without loss of generality) $\lim \frac{x^k}{\|x^k\|} = x.$ Write 
\begin{eqnarray*}
f^\infty(x) := (f_1^\infty(x), \ldots, f_n^\infty(x)) \quad \textrm{ and } \quad g^\infty(x) := (g_1^\infty(x), \ldots, g_n^\infty(x)),
\end{eqnarray*}
where $f_i^\infty$ and $g_i^\infty$ are either zero or homogeneous polynomials of degrees $d_f$ and $d_g,$ respectively. Assume that we have proved:
\begin{eqnarray}\label{PT41}
\min\{f_i^\infty(x), g_i^\infty(x)\} & = & 0 \quad \textrm{ for } \quad i = 1, \ldots, n.
\end{eqnarray}
This, of course, implies from the assumption that $x = 0.$ As $\|x\| = 1,$ we reach a contradiction.

So we are left with proving \eqref{PT41}. To see this, we first observe that
\begin{eqnarray*}
f_i^\infty(x) &=& \lim_{k \to \infty} \frac{f_i(x^k)}{\|x^k\|^{d_f}} \ = \ 0 \quad \textrm{ for } \quad i \in I,\\
g_i^\infty(x) &=& \lim_{k \to \infty} \frac{g_i(x^k)}{\|x^k\|^{d_g}} \ = \ 0 \quad \textrm{ for } \quad i \not \in I.
\end{eqnarray*}
(Recall that $f$ and $g$ are polynomial maps of degrees $d_f$ and $d_g,$ respectively.)
Take arbitrarily  $i \in I.$ If the polynomial $g_i$ is nonnegative on some subsequence of the sequence $\{x^k\}$ then $g_i^\infty(x) \ge 0$ and so 
$\min\{f_i^\infty(x), g_i^\infty(x)\} = 0.$
Otherwise, if, for all $k$ sufficiently large, $g_i(x^k) \le 0$ then $f_i(x^k) \le g_i(x^k) \le 0$ (because of $i \in I$), and hence
\begin{eqnarray*}
- \frac{1}{k} \|x^k\|  &<& f_i(x^k) \  \le \ g_i(x^k) \ \le \ 0,
\end{eqnarray*}
which, in turn, implies easily that $g_i^\infty(x) = 0,$ and hence $\min\{f_i^\infty(x), g_i^\infty(x)\} = 0.$

Similarly, we also have $\min\{f_i^\infty(x), g_i^\infty(x)\} = 0$ for all $i \not \in I,$ and so \eqref{PT41} is proved.
\end{proof}

We are now in a position to prove Theorem~\ref{Theorem13}.
\begin{proof}[Proof of Theorem~\ref{Theorem13}]
By Proposition~\ref{MD2}, the solution set $\mathrm{SOL}(f, g)$ is compact, and so it is contained in some ball with radius $R \ge 1$ centered at the origin. Consequently, \begin{eqnarray*}
\mathrm{dist}(x, \mathrm{SOL}(f, g)) & \le & 2\, \|x\| \quad \textrm{ for all } \quad \|x\| \ge R.
\end{eqnarray*}
By Lemma~\ref{BD4} and by increasing $R$ (if necessary), we can find a constant $c_1 > 0$ satisfying
\begin{eqnarray*}
c_1\, \|x\| & \le & \|\frak{m}(x)\|   \quad \textrm{ for all } \quad \|x\| \ge R.
\end{eqnarray*}
Therefore, 
\begin{eqnarray*}
\frac{c_1}{2} \, \mathrm{dist}(x, \mathrm{SOL}(f, g)) & \le & \|\frak{m}(x)\| \quad \textrm{ for all } \quad \|x\| \ge R.
\end{eqnarray*}

On the other hand, by Theorems~\ref{Theorem11} and \ref{Theorem12}, we have
\begin{eqnarray*}
c_2\, \mathrm{dist}(x, \mathrm{SOL}(f, g))^{\alpha} & \le & \|\frak{m}(x)\| \quad \textrm{ for all } \quad \|x\| \le R
\end{eqnarray*}
for some $c_2 > 0.$ Letting $c := \min\{\frac{c_1}{2}, c_2\},$ we get the desired result.
\end{proof}

\begin{remark}{\rm
Consider the function $\frak{r} \colon \mathbb{R}^n \rightarrow \mathbb{R}$ by
\begin{eqnarray*}
\frak{r}(x) &:=& \sum_{i = 1}^n \left([-f_i(x)]_+ + [-g_i(x)]_+ + \sqrt{|f_i(x) g_i(x)|} \right).
\end{eqnarray*}
Clearly, $\mathrm{SOL}(f, g) = \{x \in \mathbb{R}^n \ : \ \frak{r}(x) = 0\}.$ On the other hand, it is not hard to check that
\begin{eqnarray*}
|\min\{a, b\}| & \le & [-a]_+ + [-b]_+ + \sqrt{|ab|}
\end{eqnarray*}
for all real numbers $a, b.$ Hence $\|\frak{m}(x)\| \le \frak{r}(x)$ for all $x \in \mathbb{R}^n.$ Consequently, Theorems~\ref{Theorem11}, \ref{Theorem12} and \ref{Theorem13} still hold if we replace $\|\mathfrak{m}(x)\|$ by $\mathfrak{r}(x).$ 
}\end{remark}

We finish this section with the following result, which states that, generically, PCPs have a global Lipschitzian error bound.

\begin{proposition}
Let $d_i, d_i', i = 1, \ldots, n,$ be positive integer numbers. For some open dense semi-algebraic set of $(f, g)$ in $\mathbf{P}_{(d_1, \ldots, d_n)} \times \mathbf{P}_{(d_1', \ldots, d_n')},$ the solution set for the corresponding complementarity problem $\mathrm{PCP}(f, g)$ has a global Lipschitzian error bound:
There exists a constant $c > 0$ such that
\begin{eqnarray*}
c\, \mathrm{dist}(x, \mathrm{SOL}(f, g)) & \le & \|\frak{m}(x)\| \quad \textrm{ for all } \quad x \in \mathbb{R}^n.
\end{eqnarray*}
\end{proposition}

\begin{proof}
Let $\mathscr{U}$ be the open and dense subset of $\mathbf{P}_{(d_1, \ldots, d_n)} \times \mathbf{P}_{(d_1', \ldots, d_n')}$ 
for which the conclusion of Proposition~\ref{Proposition31} holds. Take any $(f, g) \in \mathscr{U}.$ Then $\mathrm{SOL}(f^\infty, g^\infty) = \{0\}.$
As in the proof of Theorem~\ref{Theorem13}, we can find constants $c_1 > 0$ and $R > 0$ such that
\begin{eqnarray*}
c_1\, \mathrm{dist}(x, \mathrm{SOL}(f, g)) & \le & \|\frak{m}(x)\| \quad \textrm{ for all } \quad \|x\| \ge R.
\end{eqnarray*}
So it remains to prove that
\begin{eqnarray*}
c_2\, \mathrm{dist}(x, \mathrm{SOL}(f, g)) & \le & \|\frak{m}(x)\| \quad \textrm{ for all } \quad \|x\| \le R.
\end{eqnarray*}
for some constant $c_2 > 0.$ To do this, it suffices to show that for each $x^* \in \mathrm{SOL}(f, g),$ there exist constants $c > 0$ and $\epsilon > 0$ such that
\begin{eqnarray*}
c\, \mathrm{dist}(x, \mathrm{SOL}(f, g)) & \le & \|\frak{m}(x)\| \quad \textrm{ for all } \quad \|x - x^*\| \le \epsilon.
\end{eqnarray*}

Indeed, let $x^* \in \mathrm{SOL}(f, g).$ Then there exists a subset $I$ of $\{1, \ldots, n\}$ such that the following conditions hold:
\begin{itemize}
\item[(a)] The Jacobian of the map 
$$\Phi \colon \mathbb{R}^n \rightarrow \mathbb{R}^n, \quad x \mapsto (f_i(x), g_j(x))_{i \in I, j \not \in I},$$  
at $x^*$ is non-degenerate; and
 
\item[(b)] $f_i(x^*) = 0$ and $g_i(x^*) > 0$ for all $i \in I,$ and $f_i(x^*) > 0$ and $g_i(x^*) = 0$ for all $i \not \in I.$
\end{itemize}
By the condition (a), $\Phi$ is a diffeomorphism on some neighbourhood of $x^*.$ Let $\Psi$ be its local inverse. Then the map $\Psi$ is locally Lipschitz in a neighbourhood of the origin $0 = \Phi(x^*).$ In particular, there exists a constant $c > 0$ such that for all $y$ near $0,$ 
\begin{eqnarray*}
\|y - 0\| &\ge& c\, \| \Psi(y) - \Psi(0) \| \ = \ c\, \| \Psi(y) - x^* \|.
\end{eqnarray*}
Consequently, there exists a constant $\epsilon > 0$ such that
\begin{eqnarray*}
\|\Phi(x)\| &\ge& c\, \| x - x^* \|  \quad \textrm{ for all } \quad \|x - x^*\| \le \epsilon.
\end{eqnarray*}

On the other hand, by continuity, it follows from the condition (b) that
\begin{eqnarray*}
\Phi(x) = \frak{m}(x)  \quad \textrm{ for all } \quad \|x - x^*\| \le \epsilon.
\end{eqnarray*}
(Perhaps, after reducing $\epsilon.$) Therefore, we have for all $x \in \mathbb{R}^n$ with $\|x - x^*\| \le \epsilon,$
\begin{eqnarray*}
\| \frak{m}(x)\| &\ge& c\, \| x - x^* \| \ \ge \ c\, \mathrm{dist}(x, \mathrm{SOL}(f, g)),
\end{eqnarray*}
which completes the proof.
\end{proof}

\subsubsection*{Added note}
After this paper had been completed, the authors learned that some results on PCPs (nonemptiness and compactness of the solution set, basic topological properties, and global Lipschitzian error bounds for the solution set) were obtained recently in \cite{Ling2019}.\footnote{We would like to thank Hongjin He for showing us this reference.} However, the approaches and techniques in the paper cited differ from ours, and furthermore, the following properties were not considered: genericity, uniqueness as well as error bounds with exponents explicitly determined.

\end{document}